\documentclass[oneside,english]{amsart}
\usepackage[T1]{fontenc}
\usepackage[latin9]{inputenc}
\usepackage{amstext}
\usepackage{amsthm}
\usepackage{amssymb}

\makeatletter
\numberwithin{equation}{section}
\numberwithin{figure}{section}
\theoremstyle{plain}
\newtheorem{thm}{\protect\theoremname}
\theoremstyle{plain}
\newtheorem{lem}[thm]{\protect\lemmaname}
\theoremstyle{remark}
\newtheorem{rem}[thm]{\protect\remarkname}

\numberwithin{thm}{section}
\usepackage{hyperref}

\makeatother

\usepackage{babel}
\providecommand{\lemmaname}{Lemma}
\providecommand{\remarkname}{Remark}
\providecommand{\theoremname}{Theorem}

\begin{document}
\title{Regularity for free multiplicative convolution on the unit circle}
\author{Serban T. Belinschi, Hari Bercovici, and Ching-Wei Ho}
\address{Institut de Math\textbackslash 'ematiques de Toulouse: UMR5219, Université
de Toulouse, CNRS; UPS , F-31062 Toulouse, France}
\email{Serban.Belinschi@math.univ-toulouse.fr}
\address{Mathematics Department, Indiana University, Bloomington, IN 47405,
USA}
\email{bercovic@indiana.edu}
\address{Institute of Mathematics, Academia Sinica, Taipei 10617, Taiwan; Department
of Mathematics, University of Notre Dame, Notre Dame, IN 46556, United
States }
\email{cho2@nd.edu}
\begin{abstract}
Suppose that $\mu_{1}$ and $\mu_{2}$ are Borel probability measures
on the unit circle, both different from unit point masses, and let
$\mu$ denote their free multiplicative convolution. We show that
$\mu$ has no continuous singular part (relative to arclength measure)
and that its density can only be locally unbounded at a finite number
of points, entirely determined by the point masses of $\mu_{1}$ and
$\mu_{2}$. Analogous results were proved earlier for the free additive
convolution on $\mathbb{R}$ and for the free multiplicative convolution
of  Borel probability measures on the positive half-line. 
\end{abstract}

\maketitle

\section{Introduction}

It has been known for some time that free convolutions have a strong
regularizing effect. The earliest instances of this phenomenon were
observed in \cite{v-fish1,bv-otaa,bi-iu}. For the additive case (see
\cite{v-add,bv-IU} or \cite{vdn} for definitions), it was shown
in \cite{Be-ieot,be14} that, given Borel probability measures $\mu_{1},\mu_{2}$
on $\mathbb{R}$, neither of which is a point mass, the free convolution
$\mu=\mu_{1}\boxplus\mu_{2}$ has no singular continuous part relative
to Lebesgue measure, and its density is analytic wherever positive
and finite. In addition, this density is locally bounded unless $\mu_{1}(\alpha_{1})+\mu_{2}(\alpha_{2})\ge1$
for some $\alpha_{1},\alpha_{2}\in\mathbb{R}$. The atomic part of
$\mu$ has finite support and was determined earlier \cite{Be-ieot}.
Analogous results have been obtained in \cite{ji} for the free multiplicative
convolution of Borel probability measures on $[0,+\infty)$. Despite
a strong similarity between these operations, the corresponding result
for free multiplicative convolution of Borel probability measures
on the unit circle $\mathbb{T}$ in the complex plane is still missing.
Recent results on Denjoy-Wolff points \cite[Corollary 3.3]{bbh} allow
us to rectify this omission in Theorem \ref{thm:main}.

The necessary background of subordination is given in Section \ref{sec:subord}
and the main result is proved in Section \ref{sec:Boundedness}. An
application in Section \ref{sec:An-application} yields a strengthening
of the results of \cite{indecomposable} concerning indecomposable
measures relative to free convolution.

\section{\label{sec:subord}Analytic subordination for free multiplicative
convolution}

We begin by recalling the analytical apparatus for the calculation
of free multiplicative convolutions on the unit circle $\mathbb{T}=\{z\in\mathbb{C}\colon|z|=1\}$.
An arbitrary Borel probability measure $\mu$ on $\mathbb{T}$ is
uniquely determined by its \emph{moments}
\[
\ensuremath{m_{n}(\mu)=\int_{\mathbb{T}}t^{n}\,d\mu(t)},\quad n\in\mathbb{N},
\]
and these moments are encoded in the \emph{moment generating function}
\[
\psi_{\mu}(z)=\int_{\mathbb{T}}\frac{tz}{1-tz}\,d\mu(t)=\sum_{n=1}^{\infty}m_{n}(\mu)z^{n}\,.
\]
The formal series $\psi_{\mu}$ actually converges for $z$ in the
unit disk $\mathbb{D}=\{z\in\mathbb{C}:|z|<1\}$, and
\[
\psi_{\mu}(\mathbb{D})\subset\{z\in\mathbb{C}\colon\Re z>-1/2\}.
\]
Observe that
\begin{equation}
2\Re\psi_{\mu}(z)+1=\int_{\mathbb{T}}\Re\left(\frac{\overline{\zeta}+z}{\overline{\zeta}-z}\right)\,{\rm d}\mu(t)=\int_{\mathbb{T}}\Re\left(\frac{\zeta+z}{\zeta-z}\right)\,{\rm d}\mu(\overline{\zeta}),\quad z\in\mathbb{D},\label{eq:de}
\end{equation}
and the last term above is precisely a Poisson integral. It follows
that $\mu$ can be recovered from $\psi_{\mu}$ by taking radial limits
\[
2\pi d\mu(e^{-i\theta})=\lim_{r\uparrow1}(2\Re\psi_{\mu}(e^{i\theta})+1)\,d\theta.
\]
(See for instance, \cite[Chapter 5]{akh}, \cite[Section 3]{bb-imrn},
and \cite[Chapter 1]{baf} for details.) In particular, if $\mu^{{\rm s}}$
denotes the singular part of the measure $\mu$, (\ref{eq:de}) shows
that
\begin{equation}
\lim_{r\uparrow1}\Re\psi_{\mu}(r\overline{\zeta})=+\infty\text{ for }\mu^{{\rm s}}\text{-almost all }\zeta\in\mathbb{T}.\label{eq:si}
\end{equation}
We note for further use the following consequence of (\ref{eq:de}).
\begin{lem}
\label{lem:bounded psi}If $\psi_{\mu}$ is a bounded function on
$\mathbb{D}$, then $\mu$ is absolutely continuous relative to arclength
measure and its density is bounded.
\end{lem}

Consider now two Borel probability measures $\mu_{1},\mu_{2}$ on
$\mathbb{T}=\{z\in\mathbb{C}\colon|z|=1\}$ and denote by $\mu=\mu_{1}\boxtimes\mu_{2}$
their free mutiplicative convolution. This was first defined in \cite{v-mult}
using the multplication of $^{*}$-free unitary operators and its
calculation---in case the two measures have a nonzero first moment---relied
on the analytic inverses of the functions $\psi_{\mu_{1}}$ and $\psi_{\mu_{2}}$
in the complex plane (see \cite{vdn} for the technical details).
Subsequently, Biane \cite{bi-mz} discovered that $\psi_{\mu}$ is
subordinate to $\psi_{\mu_{j}}$, $j=1,2$ in the sense of Littlewood.
This result implies that---at least when $\mu_{1}$ and $\mu_{2}$
have nonzero first moments---one can describe the function $\psi_{\mu}$
as the unique solution of a system of implicit equations. This method
for the calculation of $\psi_{\mu}$ does in fact extend to arbitrary
$\mu_{1}$ and $\mu_{2}$, as seen in \cite{bb07}. We state the result
below because it is instrumental in the proof of Theorem \ref{thm:main}.
We need the additional notation
\[
\eta_{\mu}(z)=\frac{\psi_{\mu}(z)}{1+\psi_{\mu}(z)},\quad h_{\mu}(z)=\frac{\eta_{\mu}(z)}{z}.
\]
It is easily seen that $\eta_{\mu}(\mathbb{D})\subset\mathbb{D}$,
$\eta_{\mu}(0)=0$, $\eta_{\mu}'(0)=m_{1}(\mu)$, and $h_{\mu}$ extends
to an analytic function from $\mathbb{D}$ to $\overline{\mathbb{D}}$.
If the function $h_{\mu}$ takes values in $\mathbb{T}$ then it is
constant and this happens precisely when $\mu$ is a point mass. The
following statement combines \cite[Theorem 3.2]{bb07} and \cite[Corollary 3.3]{bbh}.
\begin{thm}
\label{thm:aux}Consider Borel probability measures $\mu_{1},\mu_{2}$
on $\mathbb{T}$ and their free multiplicative convolution $\mu=\mu_{1}\boxtimes\mu_{2}$.
There exist unique continuous functions $\omega_{1},\omega_{2}:\mathbb{D}\cup\mathbb{T}\to\mathbb{D}\cup\mathbb{T}$
that are analytic on $\mathbb{D}$ and, in addition,
\begin{enumerate}
\item $\omega_{1}(0)=\omega_{2}(0)=0,$
\item $z\eta_{\mu}(z)=z\eta_{\mu_{1}}(\omega_{1}(z))=z\eta_{\mu_{2}}(\omega_{2}(z))=\omega_{1}(z)\omega_{2}(z)$,
$\omega_{1}(z)=zh_{2}(\omega_{2}(z))$, and $\omega_{2}(z)=zh_{1}(\omega_{1}(z))$
for every $z\in\mathbb{D}\cup\mathbb{T}$. In particular, $\eta_{\mu}$
extends continuously to $\mathbb{T}$. When either $\omega_{1}(z)$
or $\omega_{2}(z)$ belongs to $\mathbb{T}$, the values $\eta_{\mu_{j}}(\omega_{j}(z))$
are understood as radial limits, that is, 
\[
\eta_{\mu_{j}}(\omega_{j}(z))=\lim_{r\uparrow1}\eta_{\mu_{j}}(r\omega_{j}(z))
\]
.
\item if $m_{1}(\mu_{1})=m_{1}(\mu_{2})=0$, the functions $\eta_{\mu},\psi_{\mu},\omega_{1}$,
and $\omega_{2}$ are identically zero.
\end{enumerate}
\end{thm}

\section{Boundedness and the lack of a singular continuous part\label{sec:Boundedness}}

We are ready now to identify the singular behavior of a free multiplicative
convolution on $\mathbb{T}$. Of course, part (1) was proved in \cite{Be-ieot}.
\begin{lem}
\label{lem:atom or unbounded density}Suppose that $\mu_{1}$ and
$\mu_{2}$ are Borel probability measures on $\mathbb{T}$, neither
of which is a unit point mass, set $\mu=\mu_{1}\boxtimes\mu_{2}$,
and let $\alpha\in\mathbb{T}$.
\begin{enumerate}
\item If $\mu(\{\alpha\})>0$ then there exist $\alpha_{1},\alpha_{2}\in\mathbb{T}$
such that $\alpha_{1}\alpha_{2}=\alpha$ and $\mu_{1}(\{\alpha_{1}\})+\mu_{2}(\{\alpha_{2}\})=1+\mu(\{\alpha\})$.
\item If $\psi_{\mu}$ is unbounded near $1/\alpha$ then there exist $\alpha_{1},\alpha_{2}\in\mathbb{T}$
such that $\alpha_{1}\alpha_{2}=\alpha$ and $\mu_{1}(\{\alpha_{1}\})+\mu_{2}(\{\alpha_{2}\})\ge1$.
\end{enumerate}
\end{lem}

\begin{proof}
We only prove (2). As already mentioned, if $m_{1}(\mu_{1})=m_{1}(\mu_{2})=0$,
then $\mu$ is the Haar measure on $\mathbb{T}$, which has no singular
part and a density identically equal to $1/2\pi$. Inded, by Theorem
\ref{thm:aux}(3) $\psi_{\mu}$ is identically zero, in particular
bounded. For the remainder of the proof, we assume that at least one
of $m_{1}(\mu_{1}),m_{1}(\mu_{2})$ is non-zero, and thus the functions
$\psi_{\mu},\omega_{1},\omega_{2}$ of Theorem \ref{thm:aux} are
not constant. Suppose now that $\beta=1/\alpha$ is such that $\eta_{\mu}(\beta)=1$
or, equivalently
\[
\psi_{\mu}(\beta)=\lim_{r\uparrow1}\psi_{\mu}(r{\beta})=\infty.
\]
 Setting $\alpha_{1}=\omega_{1}(\beta)$ and $\alpha_{2}=\omega_{2}(\beta),$Theorem
\ref{thm:aux}(2) yields the equality $\alpha_{1}\alpha_{2}=\beta$.
Since $|\alpha_{j}|\le1$, it follows that in fact $\alpha_{j}\in\mathbb{T}$
for $j=1,2$. The subordination in Theorem \ref{thm:aux}(2) also
yields
\[
\ensuremath{\lim_{z\to\beta}\eta_{\mu_{j}}(\omega_{j}(z))=\eta_{\mu}(\beta)=1,\quad j=1,2,}
\]
and then 
\[
\lim_{r\uparrow1}\eta_{\mu_{j}}(r\alpha_{j})=1,\quad j=1,2,
\]
by Lindel\"{o}f's Theorem (see \cite[Theorem 2.3]{cluster}).

An application of the dominated convergence theorem shows that 
\[
\lim_{r\uparrow1}(1-r)\psi_{\mu_{j}}(r\alpha_{j})=\mu(\{1/\alpha_{j}\})\in[0,1),\quad j=1,2.
\]
In terms of the functions $\eta_{\mu_{j}}$, this amounts to 
\[
\lim_{r\uparrow1}\frac{\eta_{\mu_{j}}(r\alpha_{j})-1}{r-1}=\frac{1}{\mu_{j}(\{1/\alpha_{j}\})},\quad j=1,2,
\]
where the right-hand side is understood as $\infty$ if $\mu_{j}(\{1/\alpha_{j}\})=0$.
Using Julia-Carath\'eodory derivatives (see, for instance, \cite[Chapter I, Exercise 7]{baf})
this relation can be rewritten as $\eta_{\mu}'(\omega_{1}(\alpha))=1/(\mu_{j}(\{1/\alpha_{j}\}))$.
Properties of this derivative imply now that
\begin{align*}
\frac{1}{\mu_{1}(\{1/\alpha_{1}\})}-1 & =\liminf_{w\to\alpha_{1}}\frac{|\eta_{\mu_{1}}(w)|-1}{|w|-1}-1\\
 & =\liminf_{w\to\alpha_{1}}\frac{|\eta_{\mu_{1}}(w)|-|w|}{|w|-1}\\
\text{(substitute \ensuremath{w=\omega_{1}(z)})} & \le\liminf_{z\to\beta}\frac{|\eta_{\mu_{1}}(\omega_{1}(z))|-|\omega_{1}(z)|}{|\omega_{1}(z)|-1}\\
\text{(Theorem \ref{thm:aux}}) & =\liminf_{z\to\beta}\frac{|\omega_{1}(z)|}{|z|}\frac{|\omega_{2}(z)|-|z|}{|\omega_{1}(z)|-1}\\
 & =\liminf_{z\to\beta}\frac{|\omega_{2}(z)|-|z|}{|\omega_{1}(z)|-1}\\
 & \le\liminf_{z\to\beta}\frac{1-|\omega_{2}(z)|}{1-|\omega_{1}(z)|}.
\end{align*}
Switching the roles of $\mu_{1}$ and $\mu_{2}$ we obtain
\begin{align*}
\frac{1}{\mu_{2}(\{1/\alpha_{2}\})}-1 & \le\liminf_{z\to\beta}\frac{1-|\omega_{1}(z)|}{1-|\omega_{2}(z)|}=\left[\limsup_{z\to\beta}\frac{1-|\omega_{2}(z)|}{1-|\omega_{1}(z)|}\right]^{-1}\\
 & \le\left[\liminf_{z\to\beta}\frac{1-|\omega_{2}(z)|}{1-|\omega_{1}(z)|}\right]^{-1}\\
 & \le\left[\frac{1}{\mu_{1}(\{1/\alpha_{1}\})}-1\right]^{-1}.
\end{align*}
 A simple calculation shows now that the inequalty
\[
\left(\frac{1}{\mu_{2}(\{1/\alpha_{2}\})}-1\right)\left(\frac{1}{\mu_{1}(\{1/\alpha_{1}\})}-1\right)\le1
\]
is equivalent to $\mu_{1}(\{1/\alpha_{1}\})+\mu_{2}(\{1/\alpha_{2}\})\ge1$,
thus concluding the proof.
\end{proof}
We are now ready to state and prove the main result of this paper. 
\begin{thm}
\label{thm:main} Consider Borel probability measures $\mu_{1},\mu_{2}$
on $\mathbb{T}$ and their free multiplicative convolution $\mu=\mu_{1}\boxtimes\mu_{2}$.
Suppose that neither $\mu_{1}$ nor $\mu_{2}$ is a point mass. Then\emph{:}
\begin{enumerate}
\item The singular continuous part of $\mu$ relative to the arclength measure
is zero.
\item If
\begin{equation}
\max\{\mu_{1}(\{\alpha_{1}\})+\mu_{2}(\{\alpha_{2}\})\colon\alpha_{1},\alpha_{2}\in\mathbb{T}\}\le1,\label{eq:max<1-1}
\end{equation}
then $\mu$ is absolutely continuous relative to the arclength measure.
\item If the inequality in \emph{(\ref{eq:max<1-1})} is strict, then the
density of $\mu$ relative to the arclength measure is bounded.
\end{enumerate}
\end{thm}

\begin{rem}
It is remarkable that, for all free convolutions (see \cite{Be-ieot,ji}),
only the atomic parts of $\mu_{1},\mu_{2}$ have an impact on the
local boundedness of the density of their convolution. 
\end{rem}

\begin{proof}
The set $\{(\alpha_{1},\alpha_{2})\in\mathbb{T}^{2}:\mu_{1}(\{\alpha_{1}\})+\mu_{2}(\{\alpha_{2}\})\ge1\}$
is obviously finite. Therefore the set $S=\{\alpha\in\mathbb{T}:\eta_{\mu}(\{1/\alpha\})=1\}$
is finite as well. Since the support of the singular summand of $\mu$
is contained in $S$, it follows that this summand is a finite sum
of point masses. This proves (1). Suppose now that (\ref{eq:max<1-1})
holds. Then Lemma \ref{lem:atom or unbounded density}(1) shows that
$\mu$ is absolutely contiuous. Finally, suppose that the inequality
(\ref{eq:max<1-1}) is strict. Then Lemma \ref{lem:atom or unbounded density}(2)
implies that that $\eta_{\mu}$ does not take the value $1$ at any
point on $\mathbb{T}$. Since $\eta_{\mu}$ is continuous on $\overline{\mathbb{D}}$,
it must be bounded away from $1$. Thus $\psi_{\mu}=\eta_{\mu}/(1-\eta_{\mu})$
is a bounded function. Then (3) follows from Lemma \ref{lem:bounded psi}.
\end{proof}
\begin{rem}
Suppose that $\mu_{1}(\{\alpha_{1}\})+\mu_{2}(\{\alpha_{2}\})=1$
for some $\alpha_{1},\alpha_{2}\in\mathbb{T}.$ It was shown in \cite{Be-ieot}
that, setting $\beta_{j}=1/\alpha_{j}$ and $\beta=\beta_{1}\beta_{2}$,
we have $\omega_{j}(\beta)=\beta_{j}$ for $j=1,2$, but, of course,
$\mu(\{1/\beta\})=0$. (This can also be proved using the results
of \cite{bbh} and the `chain rule' for Julia-Carth\'eodory derivatives.)
In all computable examples, the density of $\mu$ is unbounded near
$1/\beta$. We suspect that this is true in full generality.
\end{rem}

\section{An application\label{sec:An-application}}

The following statement extends the main result of \cite{indecomposable}
for probability measures on the circle. Nearly identical proofs yield
the corresponding extensions for free additive convolution and for
free multiplicative convolution on the positive half-line. For these
two convolutions, it is not necessary to assume that one of the convolved
measures has more than two points in its support. The condition $\eta_{\mu}(\alpha)=1$
in the statement amounts to the requirement that either $\gamma$
is an atom of $\mu$, or the density of $\mu$ is unbounded near $\gamma$
(or both).
\begin{thm}
\label{thm:consecutive atoms}Consider Borel probability measures
$\mu_{1},\mu_{2}$ on $\mathbb{T}$, different from point masses,
and set $\mu=\mu_{1}\boxtimes\mu_{2}$. Suppose that $J\subset\mathbb{T}$
is an open arc such that each endpoint $\alpha$ of $J$ satisfies
$\eta_{\mu}(\alpha)=1$. If either $\mu_{1}$ or $\mu_{2}$ has more
than two points in its support, then $\mu(J)>0$.
\end{thm}

\begin{proof}
Let $\alpha$ and $\beta$ be the two endpoints of $J$, and let $\omega_{j}$
denote the subordination function of $\eta_{\mu}$ relative to $\eta_{\mu_{j}}$.
By Lemma \ref{lem:atom or unbounded density}, the points $\alpha_{j}=\omega_{j}(\alpha)$
and $\beta_{j}=\omega_{j}(\beta)$ satisfy $\mu_{1}(\{\alpha_{1}\})+\mu_{2}(\{\alpha_{2}\})\ge1$
and $\mu_{1}(\{\beta_{1}\})+\mu_{2}(\{\beta_{2}\})\ge1$. The hypothesis
implies that either $\alpha_{1}=\beta_{1}$ or $\alpha_{2}=\beta_{2}$.
Indeed, otherwise it would follow that the support of $\mu_{j}$ is
$\{\alpha_{j},\beta_{j}\}$, $j=1,2$. Switching, if necessary, the
roles of $\mu_{1}$ and $\mu_{2}$, we may assume that $\alpha_{1}=\beta_{1}$,
so $\omega_{1}(\alpha)=\omega_{1}(\beta)$. If $\mu(J)=0$, the function
$\omega_{j}$ maps $J$ to $\mathbb{T}$ injectively and $\mu_{1}(\omega_{1}(J))=0$.
Then the condition $\omega_{1}(\alpha)=\omega_{1}(\beta)$ implies
that $\omega_{1}(J)=\mathbb{T}\backslash\{\omega_{1}(\alpha)\}$,
contrary to the hypothesis that $\mu_{1}$ is not a point mass. This
contradiction yields the desired conclusion that $\mu(J)\ne0.$
\end{proof}

\end{document}